\documentclass[11pt]{amsart}
\usepackage{amsmath,amssymb,amsthm}
\usepackage[hidelinks]{hyperref}

\newcommand{\C}{\mathbb C}
\newcommand{\N}{\mathbb N}

\newcommand{\B}{\mathcal B}
\newcommand{\Ecal}{\mathcal E}
\newcommand{\Xcal}{\mathcal X}
\newcommand{\Xcpt}{\Xcal_{\mathrm c}}
\newcommand{\Xone}{\Xcal_{1+}}
\newcommand{\Tr}{\operatorname{tr}}

\newcommand{\supp}{\operatorname{supp}}
\newcommand{\dv}{\,dv}

\newcommand{\norm}[1]{\left\lVert #1\right\rVert}
\newcommand{\ip}[2]{\left\langle #1,#2\right\rangle}
\newtheorem{theorem}{Theorem}[section]
\newtheorem{proposition}[theorem]{Proposition}
\newtheorem{lemma}[theorem]{Lemma}
\theoremstyle{remark}
\newtheorem{remark}[theorem]{Remark}

\numberwithin{equation}{section}

\title[The Berger--Coburn endpoint problem]
{Critical Schatten exponents\\ for the Berger--Coburn endpoint problem}
\author{Jani A. Virtanen}
\address{University of Eastern Finland, Finland. University of Reading, UK. University of Helsinki, Finland}
\thanks{The author was supported in part by the Engineering and Physical Sciences Research Council grant EP/Y008375/1 and the Research Council of Finland grant 375631.}

\begin{document}
\begin{abstract}
Berger and Coburn showed that boundedness of a Toeplitz operator $T_g$ on the Fock space controls the heat transform $g^{(t)}$ for $1/4<t<1$, and conjectured that $T_g$ is bounded if and only if $g^{(1/4)}$ is bounded. Looi recently disproved the forward implication by constructing a bounded $T_g$ with unbounded $g^{(1/4)}$. We prove that the reverse implication also fails, even for a real-valued admissible symbol $g$ with $g^{(1/4)}\in C_0(\mathbb{C}^n)$, and that the forward implication fails even when $T_g\in S_p$ for every finite $p>1$. Our counterexamples come from Baire category and closed graph arguments, which yield no explicit symbols. Despite this two-sided failure, the forward implication has a sharp Schatten-class form: $\|g^{(1/4)}\|_\infty\le 2^n\|T_g\|_{S_1}$ for every admissible symbol $g$ with $T_g$ trace class. We show that the constant $2^n$ is optimal and that the same estimate holds with $S_p$ in place of $S_1$ for $T_g\in S_p$ when $0<p\le1$. For $p>1$, no corresponding uniform estimate holds, even for smooth compactly supported symbols. We also determine the optimal constants in the Berger--Coburn estimates for $1/4<t\le1$. Finally, for $1\le p<\infty$, we show that the equivalence $T_g\in S_p$ if and only if $g^{(1/4)}\in L^p$ holds exactly when $p=2$, answering a question of Bauer, Coburn, and Isralowitz.
\end{abstract}
\maketitle

\section{Introduction and main results}
For $n\geq1$, let $d\mu(z)=(2\pi)^{-n}e^{-|z|^2/2}\dv(z)$ on $\C^n$, where $v$ is Lebesgue measure, and write $F^2=F^2(\C^n,d\mu)$ for the Fock space. Denote by $P:L^2(d\mu)\to F^2$ the orthogonal projection and by $k_a(z)=e^{z\cdot\overline a/2-|a|^2/4}$ the normalized kernel at $a$, where $z\cdot a=\sum_j z_j a_j$. We say a measurable symbol $g$ is admissible if $gk_a\in L^2(d\mu)$ for all $a\in\C^n$.  Then $T_gf=P(gf)$ is defined on the dense set $\mathcal D_g=\{f\in F^2:gf\in L^2(d\mu)\}$, and we denote the extension simply by $T_g$. Set
\begin{equation}\label{e:heat}
    g^{(t)}(a)=(4\pi t)^{-n}\int_{\C^n}g(w)e^{-|w-a|^2/(4t)}\dv(w)
\end{equation}
whenever the integral converges absolutely. Recall that $g^{(1/2)}$ is the Berezin transform of $T_g$, while for sufficiently regular $g$ the endpoint transform $g^{(1/4)}$ is the Weyl symbol of $T_g$ under the Bargmann transform \cite[\S5]{BC94}.

Berger and Coburn proved that boundedness of $T_g$ impliesd boundedness of $g^{(t)}$ for $1/4<t<1$, and that boundedness of $g^{(t)}$ for any $0<t<1/4$ implies boundedness of $T_g$ \cite[Theorem~11 and \S5]{BC94}. They conjectured that $T_g$ is bounded if and only if $g^{(1/4)}$ is bounded \cite[\S6]{BC94}.

The conjecture holds for several classes of symbols. If $g\ge0$, then $T_g$ is bounded if and only if $g^{(1/2)}$ is bounded \cite{BC94,IZ10}. Since $g^{(1/4)}\le 2^n g^{(1/2)}$ pointwise and $g^{(1/2)}=(g^{(1/4)})^{(1/4)}$ for such $g$, both implications of the conjecture hold for nonnegative symbols. Consequently, any real-valued counterexample to either implication must change sign. For $g\in BMO^1(\C^n)$, both the conjecture and its compactness analogue hold \cite{BCI10,CIL11}. For symbols given by exponentials of complex quadratic polynomials, the reverse implication was established under suitable hypotheses in \cite{CHS19,CHSW21}, while the forward implication, together with the compactness analogue, was proved for quadratic forms in \cite{CHS24} and extended to inhomogeneous quadratic polynomials in \cite{Xiong24}.

Bauer and Fulsche \cite{BF20} revisited the Berger--Coburn estimates and proved, without pseudodifferential operators, that boundedness of $g^{(t)}$ for some $0<t<1/4$ implies boundedness of $T_g$. Dewage and Mitkovski \cite{DM25} reproved both estimates using quantum harmonic analysis \cite{Werner84,Fulsche20}.

Recently Looi \cite{L26} disproved the forward implication.  We show that the reverse implication fails even for real-valued $g$ with $g^{(1/4)}\in C_0(\C^n)$. Our first main theorem also determines precisely for which exponents $p>0$ the condition $T_g\in S_p$ guarantees boundedness of $g^{(1/4)}$ for every admissible symbol $g$.

Denote by $S_p=S_p(F^2)$ the Schatten-$p$ class with its usual norm or quasi-norm, and use the convention $S_\infty=\B(F^2)$ with the operator norm.  Let $Jf(z)=f(-z)$ and $W_af(z)=k_a(z)f(z-a)$ be parity and Weyl translation, respectively, and write $C_c^\infty(\C^n)$ for the smooth compactly supported functions on $\C^n$.

\begin{theorem}\label{thm:main}
\textup{(i)} If $g$ is admissible and $T_g\in S_p$ for some $0<p\leq1$, then
\begin{equation}\label{e:parity}
    g^{(1/4)}(a)=2^n\Tr(T_gW_aJW_a^*),\qquad a\in\C^n,
\end{equation}
and $\norm{g^{(1/4)}}_\infty\leq2^n\norm{T_g}_{S_p}$. For $p=1$, the constant $2^n$ is optimal even for $g\in C_c^\infty(\C^n)$.

\textup{(ii)} There exists a single admissible symbol $g$ with $T_g\in S_p$ for every finite $p>1$ and $g^{(1/4)}$ unbounded.  For each $1<p\leq\infty$, no constant $C$ satisfies $\norm{g^{(1/4)}}_\infty\leq C\norm{T_g}_{S_p}$ for all $g\in C_c^\infty(\C^n)$.

\textup{(iii)} There exists a real-valued admissible symbol $g$ such that $g^{(1/4)}\in C_0(\C^n)$, but $T_g$ has no bounded extension to $F^2$.
\end{theorem}

\begin{remark}
(i) Notice that we do not assume that $g\in L^1(\C^n,dv)$; see \cite[p.~572]{BC94} for an admissible symbol outside $L^1(\C^n,dv)$ whose Toeplitz operator is trace class. We emphasize that the counterexamples in~(ii) and~(iii) of the preceding theorem are nonconstructive.

(ii) Theorem~\ref{thm:category} strengthens~(ii) by showing that symbols with bounded $g^{(1/4)}$ form a meager set in the natural Fr\'echet spaces considered below, while Proposition~\ref{prop:sharp} gives the optimal constants for $1/4<t\leq1$.
\end{remark}

To state our second main theorem, denote by $\mathcal T_1(\C^n)$ the class of measurable symbols $g:\C^n\to\C$ for which
$$
    \int_{\C^n} |g(w)|^2e^{-|w-a|^2/4}\,dv(w)<\infty
$$
for all $a\in\C^n$. Notice that $\mathcal T_1$ is scritly smaller than the class of all admissable symbols. Let $1\le p<\infty$ and let $g\in\mathcal T_1$. Bauer, Coburn, and Isralowitz~\cite[Question~3]{BCI10} asked whether it holds that $T_g\in S_p$ if and only if $g^{(1/4)}\in L^p(\C^n,dv)$. For $S_p$ estimates away form the endpoint $t=1/4$, see \cite[Theorem~5.3]{DM25}. At the endpoint, we answer their question in the following theorem, which shows that the equivalence holds exactly for $p=2$. 

\begin{theorem}\label{BCI-Q3}
\textup{(i)} If $g$ is admissible, then $T_g\in S_2$ if and only if $g^{(1/4)}\in L^2$, in which case
$$
    \|g^{(1/4)}\|_2 = (2\pi)^{n/2}\|T_g\|_{S_2}.
$$
\textup{(ii)} If $p\in [1,\infty)\setminus\{2\}$, then there are real-valued symbols $g_k\in \mathcal T_1$ such that
$$
    T_{g_1}\in S_p,\ g_1^{(1/4)}\notin L^p
    \quad\text{and}\quad
    T_{g_2}\notin S_p,\ g_2^{(1/4)}\in L^p
$$
\end{theorem}

We comment briefly on the ideas behind the proofs. For $A\in S_1$, the Grossmann--Royer formula $\sigma_A(a)=2^n\Tr(AW_aJW_a^*)$ gives the Weyl symbol of $A$ \cite{Grossmann76,Royer77}, so~\eqref{e:parity} identifies $g^{(1/4)}$ with the Weyl symbol of $T_g$. This is classical for regular symbols \cite[\S5]{BC94}, \cite{Folland89}; we obtain it pointwise for every admissible $g$ with $T_g$ trace class by passing to the endpoint in the trace formula of Berger and Coburn. Evaluation of $g^{(1/4)}$ is therefore a trace pairing with $2^nW_aJW_a^*$, which has norm $2^n$ but lies in no $S_q$ with $q<\infty$, which is why $p=1$ is critical (see Remark~\ref{rem:critical}). Similarly, Theorem~\ref{BCI-Q3}~\textup{(i)} reflects the classical Weyl--Plancherel identity \cite{Pool66}, and part~\textup{(ii)} the failure of $S_p$--$L^p$ bounds for $p\ne 2$ (see Lemma~\ref{lem:obstruction}; cf.\ \cite{Simon92}). The counterexamples in Theorem~\ref{thm:main}~\textup{(ii)}--\textup{(iii)} and Theorem~\ref{BCI-Q3}~\textup{(ii)} come from Baire category and closed graph arguments. Applied to translates $(\tau_ah)(z)=h(z-a)$ of an $h\in C_c^\infty(\C^n)$ with $|a|\to\infty$, the a priori estimates these arguments would produce lose their weighted symbol seminorm terms. The density result of Lemma~\ref{lem:density} and duality then reduce the remaining estimates to bounds contradicted by explicit finite-rank or diagonal operators.

\section{Parity, density, and finite-rank tests}

Recall
\begin{equation}\label{e:kernel}
    |k_a(w)|^2d\mu(w)=(2\pi)^{-n}e^{-|w-a|^2/2}\dv(w).
\end{equation}
Thus, $W_a$ is an isometry on $L^2(d\mu)$ and direct substitution gives $W_a^{-1}=W_{-a}=W_a^*$.  Since $W_a$ and its inverse preserve entire functions, $W_a$ maps $F^2$ onto itself and commutes with $P$.  For $\tau_ag(z)=g(z-a)$, a change of variables in \eqref{e:kernel} gives $\norm{(\tau_ag)k_b}_{L^2(d\mu)}=\norm{gk_{b-a}}_{L^2(d\mu)}$, so translations preserve admissibility.  Conjugating multiplication by $g$ with $W_a$, and changing variables in the heat integral~\eqref{e:heat} give
\begin{equation}\label{e:covariance}
    T_{\tau_ag}=W_aT_gW_a^*,\qquad
    (\tau_ag)^{(t)}(z)=g^{(t)}(z-a).
\end{equation}
The operator identity holds on the natural domains and hence for bounded extensions; see also \cite[p.~576]{BC94}.

Let $e_\alpha(z)=z^\alpha/\sqrt{2^{|\alpha|}\alpha!}$, $\alpha\in\N_0^n$, be the normalized monomial basis. For $|\lambda|\leq1$, define 
\begin{equation}\label{e:E_lambda}
    E_\lambda e_\alpha=\lambda^{|\alpha|}e_\alpha
\end{equation}
with the convention $0^0=1$, and set
\begin{equation}
    \lambda_t=1-\frac1{2t},\qquad S_t=(2t)^{-n}E_{\lambda_t},\qquad t>\tfrac14.
\end{equation}
These are the operators introduced in \cite[p.~575]{BC94}. Berger and Coburn proved that, for admissible $g$ with $T_g$ bounded,
\begin{equation}\label{e:BC}
    g^{(t)}(a)=\Tr(T_gW_aS_tW_a^*),\qquad \tfrac14<t\leq\tfrac12,
\end{equation}
see \cite[Theorem~10]{BC94}.

The diagonal definition gives $\norm{E_\lambda}=1$ and $E_{-1}=J$. For $f=\sum_\alpha c_\alpha e_\alpha$ and $t\downarrow1/4$,
$$
    \norm{(E_{\lambda_t}-J)f}^2
    =\sum_\alpha\bigl|\lambda_t^{|\alpha|}-(-1)^{|\alpha|}\bigr|^2
    |c_\alpha|^2\to 0,
$$
by dominated convergence since each term is at most $4|c_\alpha|^2$. Thus,
\begin{equation}\label{e:strong-limit}
    S_t\to 2^nJ\quad\text{strongly},\qquad
    \norm{S_t}=(2t)^{-n}\leq2^n\quad(\tfrac14<t\leq\tfrac12).
\end{equation}

\begin{proof}[Proof of Theorem~\ref{thm:main}\textup{(i)}, except for sharpness]
Assume first that $T_g\in S_1$ and fix $a\in\C^n$. Set $B_t=W_a(S_t-2^nJ)W_a^*$.  By \eqref{e:strong-limit}, $B_t\to0$ strongly and $\norm{B_t}\leq2^{n+1}$.  For a finite-rank operator $F=\sum_{j=1}^r x_j\otimes y_j$, where $(x\otimes y)f=\ip{f}{y}x$, we have $\Tr(FB_t)=\sum_{j=1}^r\ip{B_tx_j}{y_j}\to0$. Given $\epsilon>0$, choose such an $F$ with $\norm{T_g-F}_{S_1}<\epsilon$. Then
$$
    |\Tr(T_gB_t)|\leq2^{n+1}\epsilon+|\Tr(FB_t)|,
$$
so $\Tr(T_gB_t)\to0$. This justifies passage to the limit in the trace in \eqref{e:BC}.

On the other hand, Cauchy--Schwarz and \eqref{e:kernel} imply
$$
    \int_{\C^n}|g(w)|e^{-|w-a|^2/2}\dv(w)
    \leq(2\pi)^n\norm{gk_a}_{L^2(d\mu)}<\infty.
$$
For $1/4<t\leq1/2$, the modulus of the integrand in \eqref{e:heat} is bounded by $\pi^{-n}|g(w)|e^{-|w-a|^2/2}$. Hence, dominated convergence shows that $g^{(t)}(a)\to g^{(1/4)}(a)$ and proves \eqref{e:parity}. Since $W_a$ and $J$ are unitary,
\begin{equation}\label{g-1/4-estimate}
    |g^{(1/4)}(a)|\leq2^n\norm{T_g}_{S_1}\norm{W_aJW_a^*}
    =2^n\norm{T_g}_{S_1}.
\end{equation}
For $0<p<1$, the singular-value inequality $\norm{A}_{S_1}\leq\norm{A}_{S_p}$ gives $S_p\subset S_1$ and the desired bound.
\end{proof}

To prove sharpness, we need the following density lemma.

\begin{lemma}\label{lem:density}
The subspace $\{T_h:h\in C_c^\infty(\C^n)\}$ is dense in $S_1$.
\end{lemma}
\begin{proof}
For $h\in L^1(dv)\cap L^\infty$, consider
\begin{equation}\label{e:antiwick}
    T_h=(2\pi)^{-n}\int_{\C^n}h(a)(k_a\otimes k_a)\dv(a),\qquad
    \norm{T_h}_{S_1}\leq(2\pi)^{-n}\norm h_{L^1(dv)}.
\end{equation}
The map $a\mapsto k_a\otimes k_a$ is continuous in $S_1$ and has trace norm one, so the integral in \eqref{e:antiwick} is a well-defined $S_1$-valued Bochner integral. For $f,u\in F^2$, its matrix coefficient is
$$
    (2\pi)^{-n}\int_{\C^n}h(a)f(a)\overline{u(a)}e^{-|a|^2/2}\dv(a)
    =\ip{hf}{u}=\ip{T_hf}{u},
$$
which verifies the identity. Pairing with any $B\in\B(F^2)$ also gives
\begin{equation}\label{e:trace-pairing}
    \Tr(T_hB)=(2\pi)^{-n}\int_{\C^n}h(a)\ip{Bk_a}{k_a}\dv(a).
\end{equation}

Since $\{T_h:h\in C_c(\C^n)\}$ is dense in $S_1$ (see \cite[Theorem~9]{BC94}), it suffices to approximate each such $T_h$ by Toeplitz operators with smooth compactly supported symbols. Fix $h\in C_c(\C^n)$ and choose $h_j\in C_c^\infty(\C^n)$ with $\norm{h_j-h}_{L^1(dv)}\to0$. By \eqref{e:antiwick},
$$
    \norm{T_{h_j}-T_h}_{S_1}
    \leq(2\pi)^{-n}\norm{h_j-h}_{L^1(dv)}\to0,
$$
which completes the proof.
\end{proof}

\begin{proof}[Sharpness in Theorem~\ref{thm:main}\textup{(i)}]
Let $P_0=e_0\otimes e_0$, so $\norm{P_0}_{S_1}=\Tr(P_0J)=1$. By Lemma~\ref{lem:density}, choose $h_j\in C_c^\infty$ with $\norm{T_{h_j}-P_0}_{S_1}\to0$.  Then $\norm{T_{h_j}}_{S_1}\to1$ and
$$
    |\Tr(T_{h_j}J)-1|\leq\norm{T_{h_j}-P_0}_{S_1}\to 0.
$$
By \eqref{e:parity}, $|h_j^{(1/4)}(0)|\to2^n$, so $\liminf_j\norm{h_j^{(1/4)}}_\infty/\norm{T_{h_j}}_{S_1}\geq2^n$, which together with~\eqref{g-1/4-estimate}, implies that $2^n$ is the optimal constant.
\end{proof}

\begin{lemma}\label{lem:tests}
For $1<p\leq\infty$ and $N\in\N$, there is an $h_N\in C_c^\infty(\C^n)$ such that
\begin{equation}\label{e:tests}
    \norm{T_{h_N}}_{S_p}<\frac32,\qquad
    |h_N^{(1/4)}(0)|>2^n\left(N^{1-1/p}-\frac12\right).
\end{equation}
\end{lemma}
\begin{proof}
Let $P_N$ be a projection onto the span of $N$ distinct monomials and put $A_N=N^{-1/p}P_NJ$, with $1/\infty=0$. Its nonzero singular values are $N^{-1/p}$, each with multiplicity $N$, so $\norm{A_N}_{S_p}=1$. Since $J^2=I$, we also have $\Tr(A_NJ)=N^{-1/p}\Tr(P_N)=N^{1-1/p}$. By Lemma~\ref{lem:density}, we can choose an $h_N$ such that $\norm{T_{h_N}-A_N}_{S_1}<1/2$. Then
$$
    \norm{T_{h_N}}_{S_p}
    \leq\norm{A_N}_{S_p}+\norm{T_{h_N}-A_N}_{S_1}<\frac32
$$
and
$$
    |\Tr(T_{h_N}J)-N^{1-1/p}|
    \leq\norm{T_{h_N}-A_N}_{S_1}\norm J<\frac12.
$$
Now \eqref{e:parity} implies \eqref{e:tests}.
\end{proof}

Since $N^{1-1/p}\to\infty$ for $p>1$, Lemma~\ref{lem:tests} shows that $\norm{g^{(1/4)}}_\infty$ cannot be bounded by a constant multiple of $\norm{T_g}_{S_p}$ on $C_c^\infty(\C^n)$. To obtain a single symbol with an unbounded transform, we combine them with translations.

\section{Symbol spaces, Baire category, and the closed graph theorem}\label{section 3}

For $m\in\N_0$ set
\begin{equation}\label{e:E}
    q_m(g)^2=(2\pi)^{-n}\int_{\C^n}|g(z)|^2
    e^{-|z|^2/2+m|z|}\dv(z),\quad
    \Ecal=\bigcap_{m\geq0}\{g:q_m(g)<\infty\}.
\end{equation}

\begin{lemma}\label{lem:symbols}
Equipped with the topology induced by the sequence of norms $q_m$, the space $\Ecal$ is Fr\'echet and coincides with the class of admissible symbols. If $|a|\leq m$, then $\norm{gk_a}_{L^2(d\mu)}\leq q_m(g)$.  Moreover, the endpoint heat transform $g^{(1/4)}$ is continuous and satisfies
\begin{equation}\label{e:loc}
    |g^{(1/4)}(a)|\leq2^n\norm{gk_a}_{L^2(d\mu)},\qquad
    \sup_{|a|\leq m}|g^{(1/4)}(a)|\leq2^nq_m(g).
\end{equation}
For $h\in C_c^\infty$ and each $m$, $q_m(\tau_ah)\to0$ as $|a|\to\infty$.
\end{lemma}
\begin{proof}
A sequence that is Cauchy in every $q_m$ has a limit in each of the corresponding weighted $L^2$ spaces.  Since $q_0\leq q_m$, all these limits agree with the $q_0$ limit. The common limit belongs to $\Ecal$, proving completeness. The topology is metrizable because the family is countable and Hausdorff because $q_0$ is a norm.

For $|a|\leq m$,
$$
    -\frac{|z-a|^2}{2}
    =-\frac{|z|^2}{2}+\operatorname{Re}(z\cdot\overline a)-\frac{|a|^2}{2}
    \leq-\frac{|z|^2}{2}+m|z|.
$$
Together with \eqref{e:kernel}, this proves the kernel-norm bound and shows that $\Ecal$ is contained in the admissible class. Conversely, compactness of the unit sphere gives finitely many unit vectors $\zeta_1,\ldots,\zeta_L$ such that $\max_j\operatorname{Re}(z\cdot\overline{\zeta_j})\geq |z|/2$ for all $z$. Indeed, the sets
$$
    \{u:|u|=1,\ \operatorname{Re}(u\cdot\overline\zeta)>1/2\},
    \qquad |\zeta|=1,
$$
form an open cover of the unit sphere, from which we can choose a finite subcover. For $m\geq1$ put $a_j=2m\zeta_j$. Then $e^{m|z|}\leq\sum_j e^{\operatorname{Re}(z\cdot\overline{a_j})}$, and hence
$$
    q_m(g)^2\leq\sum_{j=1}^L
    e^{|a_j|^2/2}\norm{gk_{a_j}}_{L^2(d\mu)}^2<\infty.
$$
The case $m=0$ follows from $k_0=1$. This proves equality of the symbol classes.  It also shows that all monomials belong to $\mathcal D_g$, since $|z^\alpha|^2\leq C_\alpha e^{|z|}$.

For the endpoint, $e^{-|w-a|^2}\leq e^{-|w-a|^2/2}$ and Cauchy--Schwarz give
\begin{align*}
    |g^{(1/4)}(a)|
    &\leq\pi^{-n}\int_{\C^n}|g(w)|e^{-|w-a|^2/2}\dv(w)\\
    &\leq\pi^{-n}
    \left(\int_{\C^n}|g(w)|^2e^{-|w-a|^2/2}\dv(w)\right)^{\tfrac12}
    \left(\int_{\C^n}e^{-|w-a|^2/2}\dv(w)\right)^{\tfrac12}\\
    &=2^n\norm{gk_a}_{L^2(d\mu)}.
\end{align*}
This proves absolute convergence and \eqref{e:loc}. For continuity, let $g_R=g\chi_{\{|z|\leq R\}}$.  Each $g_R$ is integrable and has compact support, so $g_R^{(1/4)}$ is continuous by dominated convergence. For each $m$, $q_m(g-g_R)\to0$, so \eqref{e:loc} gives $g_R^{(1/4)}\to g^{(1/4)}$ uniformly on $\{|a|\leq m\}$, and more generally, if $g_j\to g$ in $\Ecal$, then $g_j^{(1/4)}\to g^{(1/4)}$ locally uniformly on $\C^n$.

Finally, if $\supp h\subset\{|u|\leq R\}$ and $|a|>R$, then, by substituting $z=u+a$ in \eqref{e:E}, we get
$$
    q_m(\tau_ah)^2\leq(2\pi)^{-n}\norm h_{L^2(dv)}^2
    \exp\left(-\frac{(|a|-R)^2}{2}+m(|a|+R)\right)\to 0,
$$
which completes the proof.
\end{proof}

Define
$$
    \Xcal=\{g\in\Ecal:T_g\in\B(F^2)\},\qquad
    \Xcpt=\{g\in\Ecal:T_g\text{ is compact}\},
$$
$$
    \Xcal_p=\{g\in\Ecal:T_g\in S_p\}\quad(1<p<\infty),\qquad
    \Xone=\{g\in\Ecal:T_g\in\textstyle\bigcap_{p>1}S_p\}.
$$
Equip $\Xcal$ and $\Xcpt$ with $q_m(g)+\norm{T_g}$, equip $\Xcal_p$ with $q_m(g)+\norm{T_g}_{S_p}$, and equip $\Xone$ with $q_m(g)+\norm{T_g}_{S_{p_m}}$, where $p_m=1+1/(m+1)$. We write $P_m$ for these norms in each space. They are increasing because $q_m$ increases and $p_m$ decreases.  Also $\bigcap_m S_{p_m}=\bigcap_{p>1}S_p$. Indeed, for each $p>1$, choose $m$ with $p_m<p$ and use $S_{p_m}\subset S_p$. According to \eqref{e:antiwick}, $C_c^\infty$ is contained in all four types of symbol spaces.

\begin{lemma}\label{lem:complete}
Each of these symbol spaces is Fr\'echet.
\end{lemma}
\begin{proof}
First suppose that $g\in\Ecal$ and $A\in\B(F^2)$ satisfies $\ip{Ak_a}{k_b}=\ip{gk_a}{k_b}$ for all $a,b$. Then, for each $b$,
$$
    \ip{k_a}{A^*k_b}=\ip{gk_a}{k_b}
    =\ip{k_a}{P(\overline g\,k_b)},\qquad a\in\C^n.
$$
Since the span of the kernels $k_a$ is dense in $F^2$, $A^*k_b=P(\overline g\,k_b)$. Thus, for every $f\in\mathcal D_g$,
$$
    \ip{Af}{k_b}=\ip f{P(\overline g\,k_b)}
    =\ip{gf}{k_b}=\ip{P(gf)}{k_b}.
$$
Again, since the kernels span a dense subspace of $F^2$, it follows that $Af=P(gf)$ for every $f\in\mathcal D_g$. Hence, because Toeplitz operators are linear in their symbols, $\Xcal$, $\Xcpt$, $\Xcal_p$, and $\Xone$ are vector spaces.

Now let $(g_j)$ be Cauchy in every $P_m$. Lemma~\ref{lem:symbols} implies that $g_j\to g\in\Ecal$ in every $q_m$, and the operators $T_{g_j}$ converge in operator norm to some $A$. If $(g_j)\subset\Xcpt$, then each $T_{g_j}$ is compact, so $A$ is compact. In the Schatten cases the limits in the required Schatten norms also exist by completeness; they coincide with $A$ because Schatten-norm convergence implies operator-norm convergence. For every $a,b$, choose $m\geq|a|$. By Lemma~\ref{lem:symbols},
$$
    |\ip{(g_j-g)k_a}{k_b}|\leq q_m(g_j-g)\to 0,
$$
and so $\ip{Ak_a}{k_b}=\ip{gk_a}{k_b}$. The first paragraph identifies $A$ with the bounded extension of $T_g$ on $\mathcal D_g$. Thus $g$ lies in the required class and $P_m(g_j-g)\to0$ for every $m$, which proves completeness. Metrizability and the Hausdorff property follow countability and $q_0\leq P_m$.
\end{proof}

\begin{theorem}\label{thm:category}
For $Y=\Xcal$, $\Xcpt$, $\Xone$, or $\Xcal_p$ with fixed $1<p<\infty$, the set $\{g\in Y:g^{(1/4)}\in L^\infty(\C^n)\}$ is meager in $Y$.
\end{theorem}
\begin{proof}
For $M\in\N$, let $Y_M=\{g\in Y:\|g^{(1/4)}\|_\infty\leq M\}$. If $g_j\to g$ in $Y$ and $g_j\in Y_M$, then \eqref{e:loc} gives $g_j^{(1/4)}\to g^{(1/4)}$ uniformly on each compact subset of $\C^n$. Therefore $|g^{(1/4)}(a)|\leq M$ for every $a\in\C^n$, and hence $Y_M$ is closed.

Suppose that some $Y_M$ has nonempty interior. Since $P_m$ are increasing, there are $g_0\in Y_M$, $m\in\N_0$, and $\delta>0$ such that
$$
    g_0+V\subset Y_M,\qquad V=\{h\in Y:P_m(h)<\delta\}.
$$
For $h\in V$, linearity of the absolutely convergent heat integrals in \eqref{e:heat} gives $h^{(1/4)}=(g_0+h)^{(1/4)}-g_0^{(1/4)}$, so $\norm{h^{(1/4)}}_\infty\leq2M$. For arbitrary nonzero $h\in Y$, apply this to $\delta h/(2P_m(h))\in V$. It follows that
\begin{equation}\label{e:baire-bound}
    \norm{h^{(1/4)}}_\infty\leq\frac{4M}{\delta}P_m(h)\qquad(h\in Y).
\end{equation}

Let $r=\infty$ for $Y=\Xcal$ or $\Xcpt$, $r=p$ for $Y=\Xcal_p$, and $r=p_m$ for $Y=\Xone$. Apply Lemma~\ref{lem:tests} with this $r$ and, using Lemma~\ref{lem:symbols}, choose $a_N$ so that $q_m(\tau_{a_N}h_N)\leq1$. The translated symbols $\tau_{a_N}h_N$ belong to $C_c^\infty\subset Y$. By \eqref{e:covariance}, translation preserves both $\|T_{h_N}\|_{S_r}$ and $\|h_N^{(1/4)}\|_\infty$. Thus,
$$
    P_m(\tau_{a_N}h_N)<\tfrac52,\qquad
    \norm{(\tau_{a_N}h_N)^{(1/4)}}_\infty
    >2^n\left(N^{1-1/r}-\tfrac12\right).
$$
Substitution in \eqref{e:baire-bound} gives $2^n(N^{1-1/r}-\tfrac12)<10M/\delta$ for every $N$, which is a contradiction because $r>1$. Thus each closed set $Y_M$ has empty interior and is nowhere dense. Their union is exactly the set in the theorem, since the heat transforms are continuous, which proves meagerness.
\end{proof}

\begin{proof}[Proof of Theorem~\ref{thm:main} \textup{(ii)}]
First note that the space $\Xone$ is nonempty and Fr\'echet by Lemma~\ref{lem:complete}. Baire's theorem and Theorem~\ref{thm:category} imply that $\Xone\setminus\bigcup_M Y_M$ is nonempty. Any symbol $g$ in this complement is admissible, satisfies $T_g\in S_p$ for every finite $p>1$, and has unbounded endpoint transform $g^{(1/4)}$. This gives one symbol simultaneously for all these exponents.  The failure of the uniform estimates was proved using Lemma~\ref{lem:tests}.
\end{proof}

\begin{proof}[Proof of Theorem~\ref{thm:main} \textup{(iii)}]
Write $Hg=g^{(1/4)}$ and equip
$$
    \mathcal Z=\{g\in\Ecal:Hg\in C_0(\C^n)\}
$$
with the increasing family of norms
$$
    p_m(g)=q_m(g)+\|Hg\|_\infty.
$$
It is easy to see that $\mathcal Z$ is a Fr\'echet space. Indeed, a sequence that is Cauchy in every $p_m$ converges in $\Ecal$ to some $g$, while its heat transforms converge uniformly to a function in $C_0(\C^n)$. By Lemma~\ref{lem:symbols}, the same heat transforms converge locally uniformly to $Hg$, so their uniform limit is $Hg$. Hence, $g\in\mathcal Z$ and $p_m(g_j-g)\to0$ for every $m$.

Suppose that $T_g$ has a bounded extension for every $g\in\mathcal Z$. Notice the linear map
$$
    \mathcal Z\to\B(F^2), \qquad g\to T_g,
$$
has a closed graph. To see this, suppose that $g_j\to g$ in $\mathcal Z$ and $\|T_{g_j}-A\|\to 0$. For every $a\in\C^n$, convergence in $\Ecal$ implies that
$$
 	T_{g_j}k_a=P(g_jk_a)\to P(gk_a)=T_gk_a.
$$
Thus, $Ak_a=T_gk_a$ for every $a$ and hence $A=T_g$ because the span of $\{k_a\}$ is dense. By the closed graph theorem, there are $C>0$ and $m$ such that
$$
	\|T_g\|\leq C(q_m(g)+\|Hg\|_\infty) \quad\text{for all}\ g\in\mathcal Z.
$$
Therefore, for $h\in C_c^\infty(\C^n)$, using the translation $\tau_ah(z)=h(z-a)$, we get
$$
	\|T_h\|=\|T_{\tau_a h}\|\le C(q_m(\tau_ah)+\|H(\tau_a h)\|\infty)
	=C(q_m(\tau_a h)+\|Hh\|_\infty).
$$
Since $q_m(\tau_a h)\to0$ as $a\to\infty$,
\begin{equation}\label{e:converse-test-estimate}
 	\|T_h\|\leq C\|Hh\|_\infty\quad 
	\text{for all}\ h\in C_c^\infty(\C^n).
\end{equation}

Given $A\in S_1$, define $\sigma_A(a)=2^n\Tr (AW_aJW_a^*)$ for $a\in\C^n$. By~\eqref{e:parity}, $Hh=\sigma_{T_h}$ for $h\in C_c^\infty(\C^n)$, and
$$
    \|\sigma_A-\sigma_B\|_\infty\leq 2^n\|A-B\|_{S_1}
    \quad \text{for}\ A,B\in S_1.
$$
Notice that trace-norm density implies that there are $h_j\in C_c^\infty(\C^n)$ such that $\|T_{h_j}-A\|_{S_1}\to0$. Thus, by \eqref{e:converse-test-estimate},
\begin{equation}\label{e:converse-trace-estimate}
    \|A\|\leq C\|\sigma_A\|_\infty \quad \text{for all}\ A\in S_1.
\end{equation}

For $N\geq1$, let
\begin{equation}\label{e:A_N}
    u_N(z)=\frac{z_1^N}{\sqrt{2^N N!}}, \qquad e_0=1,
    \qquad A_N=u_N\otimes e_0,
\end{equation}
where $(u\otimes v)f=\langle f,v\rangle u$. Then $\|A_N\|=1$. Since
$$
    (W_aJW_a^*f)(z)
    =e^{z\cdot\bar a-|a|^2}f(2a-z),
$$
we get
\begin{equation}\label{e:sigma_A_N}
	\sigma_{A_N}(a)=2^n\langle W_aJW_a^*u_N,e_0\rangle
	=2^n e^{-|a|^2}\frac{(\sqrt{2}\,a_1)^N}{\sqrt{N!}},
\end{equation}
where $a=(a_1,\ldots, a_n)\in\C^n$. By Stirling's formula,
$$
    \|\sigma_{A_N}\|_\infty=\sup_{a\in\C^n} 2^n e^{-|a_1|^2}\frac{(\sqrt{2}\,|a_1|)^N}{\sqrt{N!}}
    =2^n\frac{N^{N/2}e^{-N/2}}{\sqrt{N!}} \sim 2^n(2\pi N)^{-1/4}\to 0,
$$
which contradicts \eqref{e:converse-trace-estimate}. Thus, there is a $g\in\mathcal Z$ for which $T_g$ has no bounded extension. 

To show that we can find such a real-valued admissable symbol, write $g=u+iv$, where $u$ and $v$ are real-valued and admissible. Since the kernel of the heat transform is real, $Hu=\operatorname{Re}(Hg)$ and $Hv=\operatorname{Im}(Hg)$ both belong to $C_0(\C^n)$. Moreover, $\mathcal D_g=\mathcal D_u\cap\mathcal D_v$. If $T_u$ and $T_v$ both have bounded extensions, their linear combination extends $T_g$, which is a contradiction. Thus, at least one of $u$ and $v$ is the desired real-valued symbol.
\end{proof}

\begin{remark}
By Theorem~\ref{thm:main}(ii), there is an admissible symbol $g$ such that $T_g$ is compact and $g^{(1/4)}$ is unbounded. Since $g^{(1/4)}$ is continuous by Lemma~\ref{lem:symbols}, it does not vanish at infinity, so the forward implication of \cite[Conjecture~2]{BCI10} fails. The symbol class $\mathcal T(\C^n)$ of \cite{BCI10} coincides with our class of admissible symbols, and the symbol in Theorem~\ref{thm:main}(iii) satisfies $g^{(1/4)}\in C_0(\C^n)$ while $T_g$ has no bounded extension. In particular, $T_g$ is not compact, so the reverse implication fails, too.

Conjectures~1 and~2 of \cite{BCI10} hold for $g\in BMO^1(\C^n)$ by \cite[Theorems~6 and~5]{BCI10}, so any counterexample to either implication lies outside $BMO^1(\C^n)$. Moreover, if $g\ge0$ and $g^{(1/4)}\in C_0(\C^n)$, then $T_g$ is compact by \cite[Corollary~1]{BCI10}, so the symbol in Theorem~\ref{thm:main}(iii) necessarily changes sign.
\end{remark}

\section{Sharp constants away from the endpoint}
\begin{proposition}\label{prop:sharp}
Let $1/4<t\leq1$ and $\kappa_n(t)=\max\{1,(4t-1)^{-n}\}$. For admissible $g$ with $T_g$ bounded, $\norm{g^{(t)}}_\infty\leq\kappa_n(t)\norm{T_g}$.
Moreover,
\begin{equation}\label{e:sharp}
    \sup_{\substack{g\in C_c^\infty(\C^n)\\T_g\neq0}}
    \frac{\|g^{(t)}\|_\infty}{\norm{T_g}}=\kappa_n(t).
\end{equation}
\end{proposition}
\begin{proof}
For $t<1$, the upper estimate follows from \cite[p.~576 and the proof of Theorem~11]{BC94}; see also \cite[Proposition~3.3(4) and the proof of Theorem~5.2]{DM25} for a proof using quantum harmonic analysis, where the heat parameter is twice ours. For $t=1$, it suffices to prove the estimate at $a=0$. Indeed, for $a\in\C^n$, the translated symbol $\tau_{-a}g$ is admissible and, by \eqref{e:covariance},
$$
    (\tau_{-a}g)^{(1)}(0)=g^{(1)}(a),\qquad
    \norm{T_{\tau_{-a}g}}=\norm{T_g}.
$$
We claim that $g^{(t)}(0)\to g^{(1)}(0)$ as $t\uparrow1$. Since $|g^{(t)}(0)|\leq\norm{T_g}$ for $1/2\leq t<1$, this immediately gives $|g^{(1)}(0)|\leq\norm{T_g}$, which proves the upper bound for $t=1$. To prove the claim, notice that for $1/2\leq t\leq1$, the modulus of the integrand in \eqref{e:heat} at $a=0$ is bounded by $(2\pi)^{-n}|g(z)|e^{-|z|^2/4}$. It remains to show that this function is integrable. Since $g$ is admissible, Lemma~\ref{lem:symbols} gives $q_1(g)<\infty$, and Cauchy--Schwarz yields
\begin{align*}
    \int_{\C^n}|g(z)|e^{-|z|^2/4}\dv(z)
    &=\int_{\C^n}
    \bigl(|g(z)|e^{-|z|^2/4+|z|/2}\bigr)e^{-|z|/2}\dv(z)\\
    &\leq(2\pi)^{n/2}q_1(g)
    \left(\int_{\C^n}e^{-|z|}\dv(z)\right)^{1/2}<\infty.
\end{align*}
Hence, dominated convergence gives $g^{(t)}(0)\to g^{(1)}(0)$ as $t\uparrow1$, and therefore $|g^{(1)}(0)|\leq\norm{T_g}$.

It remains to prove the reverse inequality in~\eqref{e:sharp}. We first express $g^{(t)}(0)$ as a trace pairing with $S_t$ for $g\in C_c^\infty(\C^n)$. For such $g$, the identity \eqref{e:BC} at $a=0$ in fact holds for every $1/4<t\leq1$. Indeed, the kernel expansion gives
\begin{equation}\label{e:Berezin-E_lambda}
    \ip{E_\lambda k_z}{k_z}
    =e^{-|z|^2/2}\sum_{\alpha\in\N_0^n}\frac{\lambda^{|\alpha|}|z^\alpha|^2}
    {2^{|\alpha|}\alpha!}
    =e^{-(1-\lambda)|z|^2/2}.
\end{equation}
Since $1-\lambda_t=1/(2t)$, \eqref{e:trace-pairing} gives
\begin{equation}\label{e:compact-trace}
    \Tr(T_gS_t)=(2\pi)^{-n}(2t)^{-n}
    \int_{\C^n}g(z)e^{-|z|^2/(4t)}\dv(z)=g^{(t)}(0).
\end{equation}

We next compute the trace norm of $S_t$. Its singular values are $(2t)^{-n}|\lambda_t|^{|\alpha|}$, $\alpha\in\N_0^n$, and therefore
$$
    \norm{S_t}_{S_1} 
    =(2t)^{-n}\sum_{\alpha\in\N_0^n}|\lambda_t|^{|\alpha|}
    =(2t)^{-n}(1-|\lambda_t|)^{-n}=\kappa_n(t).
$$
Indeed,
$$
    1-|\lambda_t|=
    \begin{cases}
    (4t-1)/(2t),&1/4<t\leq1/2,\\
    1/(2t),&1/2\leq t\leq1.
    \end{cases}
$$

We now construct symbols for which the ratio in \eqref{e:sharp} approaches $\kappa_n(t)$. Let $Q_j$ be the orthogonal projection onto the span of the monomials of total degree at most $j$, and set $U=J$ if $\lambda_t<0$ and $U=I$ if $\lambda_t\geq0$. Then $US_t=|S_t|$ and $\norm{Q_jU}=1$. Since $Q_j\uparrow I$ strongly,
$$
    \Tr(Q_jUS_t)=\Tr(Q_j|S_t|)\to \Tr|S_t|=\norm{S_t}_{S_1}=\kappa_n(t).
$$

By Lemma~\ref{lem:density}, for each $j$ there is $g_j\in C_c^\infty(\C^n)$ such that
$$
    \norm{T_{g_j}-Q_jU}_{S_1}<\frac1j.
$$
Since convergence in $S_1$ implies convergence in operator norm, $\norm{T_{g_j}}\to\norm{Q_jU}=1$. Moreover,
$$
    \bigl|\Tr(T_{g_j}S_t)-\Tr(Q_jUS_t)\bigr|
    \leq\norm{T_{g_j}-Q_jU}_{S_1}\norm{S_t}\to 0.
$$
By \eqref{e:compact-trace}, $g_j^{(t)}(0)=\Tr(T_{g_j}S_t)$, and hence $g_j^{(t)}(0)\to\kappa_n(t)$. Therefore
$$
    \liminf_{j\to\infty}\frac{\norm{g_j^{(t)}}_\infty}{\norm{T_{g_j}}}
    \geq\lim_{j\to\infty}\frac{|g_j^{(t)}(0)|}{\norm{T_{g_j}}}=\kappa_n(t).
$$
Together with the upper estimate proved above, we obtain \eqref{e:sharp}.
\end{proof}

\begin{remark}\label{rem:critical}
The preceding argument also gives an operator-theoretic interpretation of the critical exponent $p=1$. For $1/4<t\leq1$, the operator $S_t$ is the unique bounded operator $B$ such that
$$
    g^{(t)}(0)=\Tr(T_gB),\qquad g\in C_c^\infty(\C^n).
$$
Indeed, if $B$ has this property, then $\Tr(T_g(B-S_t))=0$ for every $g\in C_c^\infty(\C^n)$. By Lemma~\ref{lem:density} and continuity of the trace pairing, $\Tr(A(B-S_t))=0$ for every $A\in S_1$, and hence $B=S_t$.

At the endpoint $t=1/4$, \eqref{e:parity} shows that the unique bounded operator representing endpoint evaluation in this way is $2^nJ$. Thus, endpoint evaluation is represented by a bounded operator, which is exactly what is needed for duality with $S_1$. On the other hand, $J$ is unitary on the infinite-dimensional space $F^2$, and therefore $J\notin S_q$ for every finite $q$. Consequently, if $1<p<\infty$ and $q=p/(p-1)$ is the conjugate exponent, the operator representing endpoint evaluation does not belong to the dual Schatten class $S_q$. This reflects precisely the critical role of $p=1$ in Theorem~\ref{thm:main}.
\end{remark}

\section{Schatten classes and the endpoint heat transform}

In this section, we prove Theorem~\ref{BCI-Q3}. As in Section~\ref{section 3}, given $A\in S_1$, we write 
$$
    \sigma_A(a)=2^n\Tr (AW_aJW_a^*), \qquad a\in \C^n,
$$ 
for the Weyl symbol of $A$. According to~\cite[\S3]{Cahen18}, $A\mapsto \sigma_A$ extends from $S_1$ to a linear bijection from $S_2$ onto $L^2$,
\begin{equation}\label{e:sigma-norm}
     \|\sigma_A\|_{L^2}=(2\pi)^{n/2}\|A\|_{S_2},
\end{equation}
and
\begin{equation}\label{e:trace}
    \Tr (AB^*)=(2\pi)^{-n} \int_{\C^n} \sigma_A(a)\overline{\sigma_B(a)}\,dv(a).
\end{equation}

\begin{proof}[Proof of Theorem~\ref{BCI-Q3}]
(i) Suppose that $T_g\in S_2$. Since $S_1$ is dense in $S_2$, Lemma~\ref{lem:density} implies that there are $h_j\in C^\infty_c(\C^n)$ such that $\|T_g-T_{h_j}\|_{S_2}\to 0$. By~\eqref{e:parity} and~\eqref{e:sigma-norm},
$$
    \|h_j^{(1/4)}-\sigma_{T_g}\|_{L^2}
    =\|\sigma_{T_{h_j}-T_g}\|_{L^2}
    =(2\pi)^{n/2}\|T_{h_j}-T_g\|_{S_2}\to 0.
$$
For $1/4<t\le 1/2$ and $a\in\C^n$, by~\eqref{e:BC},
$$
    h_j^{(t)}(a)=\Tr(T_{h_j}W_aS_tW_a^*)
    \to \Tr(T_gW_aS_tW_a^*)=g^{(t)}(a).
$$
Thus, since
$$
    h^{(t)}_j(a)=(h^{(1/4)}_j)^{(t-1/4)}(a)\to (\sigma_{T_g})^{(t-1/4)}(a),
$$
we have $g^{(t)}=(\sigma_{T_g})^{(t-1/4)}$ for $1/4<t\le1/2$. Taking the limit $t\to 1/4$, we see that $g^{(1/4)}=\sigma_{T_g}\in L^2$. The norm identity follows from~\eqref{e:sigma-norm}.

Conversely, suppose that $g^{(1/4)}\in L^2$. Since $A\mapsto \sigma_A$ is a bijection from $S_2$ to $L^2$, there is an $A\in S_2$ such that $g^{(1/4)}=\sigma_A$. For $a,b\in \C^n$, let $B=k_a\otimes k_b$. By~\eqref{e:trace},
$$
    \langle Ak_a,k_b\rangle=\Tr(AB)
    =(2\pi)^{-n}\int_{\C^n} g^{(1/4)}(w)\sigma_B(w)\,dv(w).
$$
A direct computation shows that $|\sigma_B(z)|=2^ne^{-|z-c|^2}$ for $z\in\C^n$, where $c=(a+b)/2$. Thus, for $z\in\C^n$, using $|w-z|^2+|z-c|^2=2|z-\tfrac{w+c}2|^2+\tfrac{|w-c|^2}2$, we get
$$
    \int\int |g(w)|e^{-|w-z|^2}|\sigma_B(z)|\,dv(z)dv(w)
    =2^n(\pi/2)^n\int|g(w)|e^{-|w-c|^2/2}\,d(w),
$$
which can be seen to be finite by Cauchy-Schwartz. Thus, Fubini implies that
\begin{align*}
    \int_{\C^n} g^{(1/4)}(z)\sigma_B(z)\,dv(z)
    &=\int_{\C^n} g(z)\langle Bk_z, k_z\rangle\,dv(z)\\
    &=\int_{\C^n} g(z)k_a(z)\overline{k_b(z)} e^{-|z|^2/2}dv(z)\\
    &=(2\pi)^n\langle gk_a,k_b\rangle.
\end{align*}
Therefore, $\langle Ak_a,k_b\rangle = \langle gk_a, k_b\rangle$. As in the proof of Lemma~\ref{lem:complete}, this shows that $A$ is a bounded extension of $T_g$, and hence $T_g\in S_2$. 
\end{proof}

To prove (ii) of Theorem~\ref{BCI-Q3}, we use the rank-one operators $A_N=u_N\otimes e_0$ defined in~\eqref{e:A_N} and $E_\lambda$ defined in~\eqref{e:E_lambda}. We start with the following key lemma.

\begin{lemma}\label{lem:obstruction}
Let $r\in [1,\infty]\setminus \{2\}$. Then there is no constant $C$ such that
$$
    \|\sigma_B\|_{L^r} \le C \|B\|_{S_r}
$$
for all $B\in S_1$ whose Weyl symbols $\sigma_B$ are Schwartz functions.  
\end{lemma}
\begin{proof}
Suppose first that $1\le r<2$. Then, using~\eqref{e:sigma_A_N}, polar coordinates, and Stirling's formula, we get
$$
    \|\sigma_{A_N}\|^r_{L^r}
    =2^{nr}\left(\frac\pi r\right)^n
    \frac{(2/r)^{Nr/2}\,\Gamma\left(\frac{Nr}2+1\right)}{(N!)^{r/2}}\asymp N^{\frac12-\frac r4}\to \infty.
$$

Suppose next $2<r\le\infty$. For $-1<\lambda<1$, using~\eqref{e:Berezin-E_lambda}, we get
$$
    \left(\sigma_{E_\lambda}\right)^{(1/4)}(a)
    =\langle E_\lambda k_a, k_a\rangle
    =e^{-(1-\lambda)|a|^2/2}
$$
for $a\in\C^n$. Notice that
$$
    \left(e^{-\beta|\cdot|^2}\right)^{(1/4)}(a)
    =(1+\beta)^{-n}e^{-\beta|a|^2/(1+\beta)}.
$$
Choosing $\beta=(1-\lambda)/(1+\lambda)$, we see that
$$
    \sigma_{E_\lambda}(a) = \left(\frac2{1+\lambda}\right)^n e^{-\frac{1-\lambda}{1+\lambda}|a|^2}.
$$
Thus,
$$
    \frac{\|\sigma_{E_\lambda}\|_{L^r}}{\|E_\lambda\|_{S_r}} \asymp (1+\lambda)^{-n(1-2/r)}\to\infty
$$
as $\lambda\to -1$. It remains to note that the Weyl symbols above are all Schwartz functions.
\end{proof}

\begin{proof}[Proof of Theorem~\ref{BCI-Q3}]
(ii) For $m\ge 0$ and $g\in\mathcal T_1$, define
$$
    q^1_m(g) = q_m\left(e^{|\cdot|^2/8}g\right)
    \quad\text{and}\quad
     \Ecal^1=\bigcap_{m\geq0}\{g:q^1_m(g)<\infty\},
$$
where $q_m$ and $\Ecal$ were given in~\eqref{e:E}. Similarly to Lemma~\ref{lem:symbols}, $\Ecal^1$ is Fr\'echet, and it embeds continuously into $\Ecal$. For $p\in [1,\infty)\setminus\{2\}$, define
$$
    \Xcal^1_p=\{g\in\Ecal^1:T_g\in S_p\}
    \quad\text{and}\quad
    \mathcal Z^1_p=\{g\in\Ecal^1:g^{(1/4)}\in L^p\}
$$
which we equip with the norms $q^1_m(g)+\|T_g\|_{S_p}$ and $q^1_m(g)+\|g^{(1/4)}\|_{L^p}$, respectively. Similarly to Lemma~\ref{lem:complete}, these are Fr\'echet spaces.

We now show existence of $g_1$ and $g_2$ using the previous lemma. Suppose first that $g^{(1/4)}\in L^p$ whenever $g\in \Xcal^1_p$. By local uniform convergence, $g\mapsto g^{(1/4)}$ has a closed graph, and hence the closed graph theorem implies that there are $C$ and $m$ such that 
$$
    \|g^{(1/4)}\|_{L^p} \le C(q^1_m(g)+\|T_g\|_{S_p})
$$
Applying this inequlity to the translation $\tau_a h$, where $h\in C^\infty_c$, and using~\eqref{e:covariance}, we obtain
$$
    \|h^{(1/4)}\|_{L^p} = \|(\tau_a h)^{(1/4)}\|_{L^p}
    \le C(q^1_m(\tau_a h)+\|T_h\|_{S_p})\to C\|T_h\|_{S_p}
$$
as $|a|\to\infty$. Now if $B\in S_1$, then, by Lemma~\ref{lem:density}, there are $h_j\in C^\infty_c$ such that $\|T_{h_j}-B\|_{S_p}\le\|T_{h_j}-B\|_{S_1}\to 0$, and so, 
$$
    \|h^{(1/4)}_j-\sigma_B\|_\infty
    =\|\sigma_{T_{h_j}}-\sigma_B\|_\infty
    \le 2^n\|T_{h_j}-B\|_{S_1} \to 0.
$$
By Fatou's lemma,
$$
    \|\sigma_B\|^p_{L^p}\le \liminf_j \|h^{(1/4)}_j\|^p_{L^p}
    \le C^p\lim_j\|T_{h_j}\|^p_{S_p}=C^p\|B\|^p_{S_p},
$$
which contradicts Lemma~\ref{lem:obstruction}.

Suppose that $\mathcal Z^1_p\subset \Xcal^1_p$. Similarly to the first half 0f (ii), using the translation and the closed graph theorem, there is a constant such that
$$
    \|T_h\|_{S_p}\le C\|h^{(1/4)}\|_{L^p}
$$
for all $h\in C^\infty_c$. Denote by $p'$ the conjugate exponent of $p$ and let $B\in S_1$ be such that its Weyl symbol $\sigma_B$ is a Schwartz function. Then
\begin{align*}
    \left|\int_{\C^n} h^{(1/4)} \sigma_B\,dv\right|
    &=(2\pi)^n|\Tr(T_hB)|
    \le (2\pi)^n\|T_h\|_{S_p}\|B\|_{S_{p'}}\\
    &\le (2\pi)^n C\|h^{(1/4)}\|_{L^p}\|B\|_{S_{p'}}.
\end{align*}
By Wiener's Tauberian theorem (see, e.g., \cite[\text{Theorem 7.1 (1)}]{LS18}) and standard convolution approximation, the functions $h^{(1/4)}$ with $h\in C_c^\infty(\C^n)$ are dense in $L^p(\C^n)$. Thus, the previous estimate holds for all $h\in L^p$ and hence
$$
    \|\sigma_B\|_{L^{p'}}
    =\sup_{\|f\|_{L^p}\le 1}\left|\int_{\C^n} f \sigma_B\,dv\right|
    \le (2\pi)^nC\|B\|_{S_{p'}},
$$
which contradicts Lemma~\ref{lem:obstruction}.

Finally, the counterexamples can be chosen real-valued. Indeed, if $g=u+iv$ with $u,v$ real-valued, then $u,v\in\mathcal E^1$ and $g^{(1/4)}=u^{(1/4)}+iv^{(1/4)}$. For the reverse implication, the argument used in the proof of Theorem~\ref{thm:main}~(iii) applies unchanged. For the forward implication, if $T_g$ extends to $A\in S_p$, then $T_{\bar g}$ extends to $A^*$. Thus, $T_u$ and $T_v$ have extensions $(A+A^*)/2$ and $(A-A^*)/(2i)$, respectively, both in $S_p$. Since $g^{(1/4)}\notin L^p$, at least one of $u^{(1/4)}$ and $v^{(1/4)}$ does not belong to $L^p$.
\end{proof}

\end{document}